\documentclass[12pt,a4paper,oneside,english]{amsart}
\usepackage[T1]{fontenc}
\usepackage[latin9]{inputenc}
\usepackage{xcolor}
\usepackage{verbatim}
\usepackage{amsthm}
\usepackage{amssymb}
\PassOptionsToPackage{normalem}{ulem}
\usepackage{ulem}

\makeatletter

\providecommand{\tabularnewline}{\\}
\providecolor{lyxadded}{rgb}{0,0,1}
\providecolor{lyxdeleted}{rgb}{1,0,0}

\theoremstyle{plain}
\newtheorem{thm}{Theorem}[section]
  \theoremstyle{definition}
  \newtheorem{defn}[thm]{Definition}
  \theoremstyle{plain}
  \newtheorem{prop}[thm]{Proposition}
  \theoremstyle{plain}
  \newtheorem{lem}[thm]{Lemma}
  \theoremstyle{remark}
  \newtheorem{rem}[thm]{Remark}
  \theoremstyle{plain}
  \newtheorem{cor}[thm]{Corollary}
 \theoremstyle{definition}
  \newtheorem{example}[thm]{Example}

\usepackage{amsfonts}

\usepackage{amsthm}
\usepackage[all]{xy}
\usepackage{amscd}

\renewcommand{\hom}{\mathrm{Hom}}

\newcommand{\Ad}{\mathop{\mathrm{Ad}}\nolimits}

\def\lie{\mathsf{Lie}}
\def\quot{/\!\!/}

\numberwithin{equation}{section} \numberwithin{teorema}{section}

\makeatother

\usepackage{babel}

\begin{document}

\title[Stability and Irreducibility in Reductive Groups]{Stability of Affine $G$-varieties and Irreducibility in Reductive
Groups}

\author{Ana Casimiro and Carlos Florentino}
\begin{abstract}
Let $G$ be a reductive affine algebraic group, and let $X$ be an
affine algebraic $G$-variety. We establish a (poly)stability criterion
for points $x\in X$ in terms of intrinsically defined closed subgroups
$H_{x}$ of $G$, and relate it with the numerical criterion of Mumford,
and with Richardson and Bate-Martin-Röhrle criteria, in the case $X=G^{N}$.
Our criterion builds on a close analogue of a theorem of Mundet and
Schmitt on polystability and allows the generalization to the algebraic
group setting of results of Johnson-Millson and Sikora about complex
representation varieties of finitely presented groups. By well established
results, it also provides a restatement of the non-abelian Hodge theorem
in terms of stability notions.
\end{abstract}


\maketitle

\section{Introduction and Main Results}

Let $G$ be a connected complex affine reductive group, and $\Gamma$
a finitely presented group. The $G$-representation variety of $\Gamma$,
defined as $X:=\hom(\Gamma,G)$, is a complex affine $G$-variety
under the canonical conjugation action of $G$ on $X$. In the context
of Geometric Invariant Theory, and building on earlier work by Johnson-Millson,
Richardson and others (see \cite{JM,Rich}), a recent article of Sikora
\cite{Sikora} establishes a close relationship between stability
properties of a point $x\in X$ and the property that the image of
the representation, $x(\Gamma)\subset G$, is irreducible or completely
reducible as a subgroup of $G$.

In this article, we generalize these relationships to a bigger class
of affine $G$-varieties, where $G$ is an affine reductive group,
not necessarily irreducible, defined over an algebraically closed
field $k$, of characteristic zero. 

Besides its intrinsic relevance as stability criteria, the constructions
examined here perfectly agree with other known results for certain
specific classes of $G$-varieties. For example, the non-abelian Hodge
theorem (see \cite{GGM} or \cite{We}), in the case of complex reductive
groups, can be restated as a correspondence between stable points
in distinct (however homeomorphic) varieties. Also, from our set up,
one can recover the Mundet-Schmitt criterion for polystability \cite{MR}.

To describe our main results, let $Y(G)$ denote the set of one parameter
subgroups (1PS for short) of $G$, that is, homomorphisms $\lambda$
from the multiplicative group $k^{\times}$ to $G$. Given a 1PS $\lambda\in Y(G)$,
and $g\in G$, the morphism $t\mapsto\lambda(t)g\lambda(t)^{-1}$
may or may not extend to $k$ (as a morphism of affine varieties).
In case there is such an extension, we say that $\lim_{t\to0}\lambda(t)g\lambda(t)^{-1}$
exists. It is known that\[
P(\lambda)=\{g\in G:\lim_{t\to0}\lambda(t)g\lambda(t)^{-1}\mbox{ exists}\}\]
is always a parabolic subgroup of $G$.

Let $X$ be an affine variety with an action of $G$, that is, an
\emph{affine $G$-variety}. Recall that a point $x\in X$ is called
\emph{polystable }if its $G$-orbit is closed. Let $G_{x}$ denote
the stabilizer of $x\in X$, and consider the normal subgroup \[
G_{X}:=\bigcap_{x\in X}G_{x}\subset G.\]
A point $x\in X$ will be called \emph{stable} if it is polystable
and the quotient $G_{x}/G_{X}$ is finite. Note that there are many
slightly different notions of stability in the literature (see \cite{MumFoKir,N,Rich}),
but all of them coincide when $G_{X}$ is finite. We also define $x$
to be \emph{equicentral} if $x$ is polystable and $G_{x}=G_{X}$.

We denote by $\Lambda_{x}$ the subset of $Y(G)$ consisting of 1PS
$\lambda$ such that $\lim_{t\to0}\lambda(t)\cdot x$ exists (where
$g\cdot x$ denotes the action of $g\in G$ on $x\in X$, see Section
3 for details). We say that a subset $\Lambda\subset Y(G)$ is \emph{symmetric}
if given any $\lambda\in\Lambda$, there is another 1PS $\lambda'\in\Lambda$
such that $P(\lambda)\cap P(\lambda')$ is a Levi subgroup of both
$P(\lambda)$ and $P(\lambda')$. Following the constructions in Mundet-Schmitt
\cite{MR}, one can show (See Theorem \ref{thm:poly-symm}).
\begin{thm}
\label{thm:M-S}Let $G$ be a reductive algebraic group and $X$ be
an affine $G$-variety, both defined over $k$. Then, a point $x\in X$
is polystable if and only if $\Lambda_{x}$ is symmetric.
\end{thm}
Consider now another natural construction. Given $x\in X$, define
the closed subgroup of $G$, \[
H_{x}:=\bigcap_{\lambda\in\Lambda_{x}}P(\lambda).\]
Recall that a closed subgroup $H$ of $G$ is called \emph{irreducible}
if it is not contained in a proper parabolic subgroup of $G$, and
it is called \emph{completely reducible} if, for any inclusion of
$H$ in a parabolic $P$ of $G$, there is a Levi subgroup $L$ of
$P$ such that $H\subset L$. These natural notions, generalizing
the well known definitions for the general linear group (see \cite{Serre}),
can be extended to algebraic groups which are not necessarily connected
(see Section 2, below).

Our main result applies to certain affine $G$-varieties where the
existence of limits under one parameter subgroups in $X$ is related
to the existence of limits under conjugation in $G$. More precisely,
consider the following condition for every pair $x\in X$ and $\lambda\in Y(G)$:\begin{equation}
\mbox{if }H_{x}\subset P(\lambda)\mbox{, then }\lambda\in\Lambda_{x}.\label{eq:proper}\end{equation}
 The main result can be presented as follows (see Theorems \ref{thm:polyCR}
and \ref{thm:1PSstableIrred}). Let $Z(G)$ denote the center of $G$. 
\begin{thm}
\label{thm:Main}Let $G$ be reductive and $X$ be an affine $G$-variety,
as above. Suppose condition (\ref{eq:proper}) is satisfied. Then:

(1) the subgroup $H_{x}$ is completely reducible if and only if $x$
is polystable;

(2) if $G_{X}=Z(G)$, then $H_{x}=G$ if and only if $x$ is stable.
\end{thm}
The proof of (1) above follows similar arguments in Mundet-Schmitt
(see\cite{MR}), which in turn rely on several techniques available,
such as the strengthening of Hilbert-Mumford criterion by Birkes/Richardson
and Kempf \cite{Kempf78,Birkes}. The proof of (2) uses the notion
of 1PS stability: a point $x$ is called 1PS stable if $\lambda\in\Lambda_{x}$
implies $\lambda\in Y(G_{X})\subset Y(G)$. Although this looks more
general than the notion of stability above, we show that the two definitions
agree (see Theorem \ref{thm:new-Mumford-crit}).

Suppose now $X$ is a closed $G$-invariant subvariety of $G^{N}$,
where $G$ acts on $G^{N}$ diagonally by conjugation. For a given
$x\in X$, let $\phi_{x}$ be the subgroup of $G$ generated by the
elements $f_{1}(x),...,f_{N}(x)\in G$, where $f=(f_{1},...,f_{N}):X\hookrightarrow G^{N}$
is the natural inclusion. For this class of $G$-varieties $X$, condition
(\ref{eq:proper}) is automatically satisfied (although $G_{X}\neq Z(G)$
in general). 

Our setup provides another approach to results of Richardson (see
\cite[Theorems 3.6 and 4.1]{Rich}), Martin and Bate-Martin-Röhrle
(see \cite[Proposition 2.13 and Corollary 3.7]{Martin,BMR}) relating
stability (resp. polystability) of $x\in X$ with irreducibility (resp.
complete reducibility) of $\phi_{x}$. Together with a corresponding
relation for equicentral points of $X$, they can be stated as follows
(see Theorem \ref{thm:polyCRrep-1}). We say that a subgroup $H\subset G$
is \emph{isotropic} if it is completely reducible and its centralizer
equals $Z(G)$. An isotropic subgroup is always irreducible, but not
conversely (see Section \ref{sec:appen} in the Appendix).
\begin{thm}
\label{thm:Main-Rich}Let $X$ be a closed $G$-invariant subvariety
of $G^{N}$, as above. Then,

(1) A point $x$ is stable (resp. polystable) if and only if $\phi_{x}$
is an irreducible (resp. completely reducible) subgroup of $G$.

(2) If $G_{X}=Z(G)$, then $x$ is equicentral if and only if $\phi_{x}$
is isotropic.
\end{thm}
This result also generalizes to the algebraic group setting the recent
result of Sikora (\cite[Thm. 29, Cor. 31]{Sikora}), as representation
varieties over $\mathbb{C}$ of finitely generated groups are a particular
class of closed $G$-invariant subvarieties of $G^{N}$. As a corollary
of Theorems \ref{thm:Main} and \ref{thm:Main-Rich}, we can write
the relation between $H_{x}$ and $\phi_{x}$ as follows:
\begin{thm}
\label{thm:Hx-phix}Let $X$ be a closed $G$-invariant subvariety
of $G^{N}$ as above. Then $H_{x}$ is completely reducible if and
only if $\phi_{x}$ is completely reducible. Moreover, assuming $G_{X}=Z(G)$,
then $H_{x}=G$ if and only if $\phi_{x}$ is irreducible.
\end{thm}
Another consequence of Theorem \ref{thm:Main-Rich} is that we can
restate one form of the non-abelian Hodge theorem (see, for example
\cite{GGM}) as a relation between stability of $G$-Higgs bundles
over a Riemann surface and stability of representations in the character
variety of its fundamental group, in a framework also valid for reductive
groups which are not necessarily connected.

The article is organized as follows. In section 2, after introducing
the basic definitions, we describe some properties of irreducible
and completely reducible subgroups of a given affine reductive group,
and in section 3, we provide the stability definitions we will need,
present relations between them and prove the relevant versions of
the Hilbert-Mumford numerical criterion. In section 4, we prove our
main results, and re-derive the Mundet-Schmitt criterion for polystability.
We also analyse condition (\ref{eq:proper}) showing that it holds
for the adjoint representation of $G$ (in the complex case). In section
5, we apply our results to the setting of closed subvarieties of $G^{N}$
and of $G$-representation varieties of finitely presented groups,
adding to the work of Richardson and Sikora alluded above. We also
briefly recall $G$-Higgs bundles theory (in the complex case) in
order to present a restatement of the non-abelian Hodge theorem.

In the appendix, we collect some results that are useful in comparing
the distinct notions of irreducibility and related properties that
appear in the literature.

\section{Irreducible and Completely Reducible Subgroups}

We start by recalling some important definitions and properties of
the main objects considered, and by fixing terminology.

Let $G$ be an affine algebraic group (not necessarily irreducible)
over an algebraically closed field $k$ of characteristic zero, and
let $G^{0}$ be the connected component of the unit element of $G$.
The unipotent radical of $G$, denoted by $R_{u}(G)$, is the maximal
connected unipotent normal subgroup of $G$. 

An affine algebraic group $G$ is called \emph{reductive} if the unipotent
radical of $G^{0}$ is the trivial group%
\footnote{Note that, contrary to many references, we do not assume that a reductive
group is itself connected.%
}. In the representation theory of reductive groups, parabolic and
Levi subgroups play an important role. Recall that a \emph{parabolic
subgroup} $P$ of $G$ is a closed subgroup such that the coset space
$G/P$ is a complete variety and a \emph{Levi subgroup} of $G$ is
a connected subgroup isomorphic to $G/R_{u}(G)$.

Note that, by definition, any parabolic subgroup $P$ of $G$ is identified
with the semidirect product $P=L\ltimes R_{u}(P)$, where $L$ is
a Levi subgroup of $P$.

We will use the notation $Z(H)$ for the center of a group $H$, and
when $H$ is a subgroup of $G$, $Z_{G}(H)$ will denote the centralizer
of $H$ in $G$.

\subsection{Parabolic and one parameter subgroups}

In practice, we need another characterization of parabolic subgroups
for which we recall the definition of one-parameter subgroup. Let
$k^{\times}$ be the multiplicative group of invertible elements of
the field $k$. 

A morphism of algebraic groups $\lambda:k^{\times}\to G$ is called
a one parameter subgroup (1PS), or cocharacter of $G$. The trivial
1PS is the morphism given by $\lambda(t)=e,\ \forall t\in k^{\times}$,
where $e\in G$ is the identity element. The set of cocharacters of
$G$ will be denoted by $Y(G)$.

%
{}
\begin{defn}
\label{def:rays}Given a 1PS $\lambda\in Y(G)$ and an element $g\in G$,
consider the morphism $\lambda_{g}:k^{\times}\to G$, defined by \[
\lambda_{g}(t):=\lambda(t)g\lambda(t)^{-1},\quad t\in k^{\times}.\]
We say that $\lim_{t\to0}\lambda_{g}(t)$ exists when $\lambda_{g}$
can be extended to a morphism $\tilde{\lambda}_{g}:k\to G$ (that
is, we have $\lambda_{g}=\tilde{\lambda}_{g}\circ i$, for the obvious
inclusion $i:k^{\times}\to k$). In this case we write \[
\lambda^{+}g=\lim_{t\to0}\lambda_{g}(t)\]
for the element $\tilde{\lambda}_{g}(0)\in G$, which is uniquely
defined. 
\end{defn}
We follow the convention that, whenever a formula with limits is written,
we are assuming that they exist.
\begin{defn}
\label{def:parab}Given a 1PS $\lambda\in Y(G)$, define the following
subsets of $G$:\begin{align*}
P(\lambda) & :=\left\{ g\in G:\ \lambda^{+}g\mbox{ exists }\right\} \\
U(\lambda) & :=\left\{ g\in G:\ \lambda^{+}g=e\right\} \\
L(\lambda) & :=\left\{ g\in G:\ \lambda^{+}g=g\right\} .\end{align*}
We call $P(\lambda)$ a \emph{R-parabolic} subgroup of $G$, $U(\lambda)$
a \emph{R-unipotent} subgroup of $P(\lambda)$ and $L(\lambda)$ a
\emph{R-Levi} subgroup of $P(\lambda)$.
\end{defn}
The terminology R-parabolic and R-Levi was introduced in \cite{BMR}
referring to its relevance in the work of Richardson. For a 1PS $\lambda\in Y(G),$
$\lambda(k^{\times})$ denotes the subgroup $\lambda(k^{\times}):=\{\lambda(t):t\in k^{\times}\}\subset G$.
When $G$ is reductive, the following is well known.
\begin{prop}
Let $G$ be an affine reductive group and let $\lambda\in Y(G)$.
Then,

(i) $P(\lambda)$ is a parabolic subgroup of $G$, $L(\lambda)$ is
a Levi subgroup of $P(\lambda)$, and $U(\lambda)$ is the unipotent
radical of $P(\lambda)$; in particular $P(\lambda)=L(\lambda)\ltimes U(\lambda)$.

(ii) $L(\lambda)$ coincides with the centralizer of the subgroup
$\lambda(k^{\times})$.

(iii) In the case that $G$ is connected, then all parabolic subgroups
of $G$ are $R$-parabolic subgroups.\end{prop}
\begin{proof}
See \cite{MumFoKir} or \cite[Prop. 8.4.5 and Thm. 13.4.2]{Springer}. 
\end{proof}
Note that $P(\lambda)=G$ if and only if $\lambda(k^{\times})$ is
contained in the center of $G$, $Z(G)$. To rule out this trivial
case, the parabolics $P(\lambda)\neq G$ will be called proper R-parabolics.

{}

\subsection{\label{sub:Irreducible-and-completely}Irreducible and completely
reducible subgroups}

Assume, from now on that $G$ is an affine reductive group. 
\begin{defn}
Let $H$ be a subgroup of $G$.
\begin{enumerate}
\item We say that $H\subset G$ is \emph{irreducible }(in\emph{ }$G$) if
$H$ is not contained in any proper $R$-parabolic subgroup of $G$.
\item We say that $H\subset G$ is \emph{completely reducible} (in\emph{
}$G$) if for every R-parabolic subgroup $P$ which contains $H$
there is a R-Levi subgroup $L$ of $P$, such that $H\subset L$.
\end{enumerate}
\end{defn}
Note that in particular, an irreducible subgroup is completely reducible.
These definitions are given in \cite[§6]{BMR} and coincide with the
original definition of Serre (\cite{Serre}), when $G$ is connected;
they were introduced to generalize standard methods in the representation
theory of $GL(V)$ to arbitrary reductive algebraic groups. 

%
{}
\begin{rem}
\label{rem:closure}Observe that a subgroup $H\subset G$ is irreducible
(resp. completely reducible) if and only if $\overline{H}$, its Zariski
closure in $G$, is irreducible (resp. completely reducible). Indeed,
this easily follows from the fact that any R-parabolic subgroup of
$G$ is Zariski closed in $G$.
\end{rem}
Here, and except when explicitly mentioned otherwise, all topological
notions on algebraic varieties will refer to the Zariski topology.

\section{Stability Notions for Affine $G$-varieties}

In this section we study the natural notions of (poly)stability that
turn out to be most useful for the statement of Theorem \ref{thm:Main}.
Here, $X$ denotes a $G$-variety and $G$ an affine reductive group
(neither of which is assumed to be irreducible).

\subsection{\label{sub:Stability,proper-stability}Stability, proper stability
and polystability}

The structure of $G$-variety on $X$ assumes the existence of a morphism
satisfying the usual axioms for an action\[
\psi:G\times X\to X\]
and will be denoted by $(g,x)\mapsto g\cdot x$, ($g\in G$, $x\in X$)
where no confusion arises. The orbit space $X/G$ is generally not
an algebraic set. However, since $G$ is reductive, there exists a
categorical quotient $X\quot G$ which is an affine algebraic variety.
In fact, this quotient is also the so-called \emph{affine quotient}
defined as the spectrum of the ring $k[X]^{G}$ of regular $G$-invariants
inside the coordinate ring $k[X]$. It can also be shown that the
affine quotient parametrizes the set of closed $G$-orbits in $X$.%
{}

For a given point $x\in X$, denote by $O_{x}$ or by $G\cdot x$
its $G$-orbit, and by $G_{x}$ its stabilizer in $G$. 
\begin{defn}
The closed subgroup $G_{X}:=\bigcap_{x\in X}G_{x}\subset G$ will
be called the \emph{center of the action} of $G$ on $X$, or simply
the \emph{center of $X$} when the action is understood. When $G_{X}$
is not trivial, it makes sense to consider the \emph{normalized stabilizer},
defined as $NG_{x}:=G_{x}/G_{X}$.\end{defn}
\begin{rem}
Note that for any $G$-variety, $G_{X}$ is a normal subgroup of $G$.
Indeed, for any $x\in X$ the stabilizers verify $gG_{x}g^{-1}=G_{g\cdot x}$
for all $g\in G$.
\end{rem}
Let $\hat{G}=G/G_{X}$ denote the space of (left) cosets of $G_{X}$,
which is again an affine reductive group (see Borel \cite[6.8 Theorem]{Bo}).
Then, the $G$-variety $X$ has also the structure of a $\hat{G}$-variety,
since the action $\psi$ factors through a map (henceforth called
the \emph{reduced action})\[
\hat{\psi}:\hat{G}\times X\to X\]
defined by $(gG_{X},x)\mapsto g\cdot x$ (well defined since $G_{X}$
is normal). 

It is clear that, as affine algebraic varieties, both quotients $X\quot G$
and $X\quot\hat{G}$ are isomorphic. This idea has been explored in
affine GIT, and because of this, one usually considers the reduced
action $\hat{\psi}$. 

In contrast, in this article, we want to consider the original action
$\psi$ because it turns out to be more adapted to relate with the
concept of irreducibility in reductive groups. 

More abstractly, this construction can be seen as a correspondence
from the category of $G$-spaces to the category of $G/G_{X}$-spaces,
and we want to understand how stability behaves under this correspondence.
For this, we will use the following definitions of stability closely
related to the original definitions given in \cite{MumFoKir}. Recall
that we are using the Zariski topology.
\begin{defn}
\label{def:polystab}Let $X$ be a $G$-variety and $x$ be a point
in $X$.
\begin{enumerate}
\item We say that $x$ is \emph{polystable} if $O_{x}=G\cdot x$ is closed;
\item We say that $x$ is \emph{properly stable} if it is polystable and
$G_{x}$ is finite; 
\item We say that $x$ is\emph{ stable} if it is polystable and $NG_{x}$
is finite;
\item We say that $x$ is \emph{equicentral} if it is polystable and $NG_{x}$
is trivial, that is $G_{x}=G_{X}$.
\end{enumerate}
\end{defn}
\begin{rem}
As far as we known, the definition of stability in (3) was introduced
by Richardson (\cite{Rich}) and coincides with proper stability (as
defined by Mumford), when $G_{X}$ is finite. Note also that, if either
$x\in X$ is properly stable, or $x$ is equicentral, then $x$ is
stable.
\end{rem}
Let $X$ be an affine $G$-variety, and consider, for all $x\in X$,
the orbit morphism, defined as \[
\psi_{x}:G\to X,\quad g\mapsto g\cdot x.\]
It is clear that $\psi_{x}$ is an affine morphism and that its image
$\psi_{x}(G)$ coincides with the orbit of $x$, $G\cdot x\subset X$.
In the same way as proper stability is equivalent to properness of
the orbit morphism (see Newstead \cite[Lemma 3.17]{N}), stability
is equivalent to properness of the reduced orbit morphism.
\begin{prop}
\label{pro:stableproper}Let $X$ be an affine $G$-variety. Then,
$x\in X$ is stable if and only if the reduced orbit morphism $\hat{\psi}_{x}:G/G_{X}\to X$,
$\hat{\psi}_{x}(gG_{X}):=g\cdot x$ is proper.\end{prop}
\begin{proof}
It is clear that $\hat{\psi}_{x}$ is well defined. If $\hat{\psi}_{x}$
is proper, then the image of $\hat{\psi}_{x}$ is closed. First we
show that $\hat{\psi}_{x}(G/G_{X})=G\cdot x$, the inclusion $\hat{\psi}_{x}(G/G_{X})\subset G\cdot x$
being clear. Note that $\hat{\psi}_{x}$ factors through the canonical
projection $\pi_{x}:G/G_{X}\rightarrow G/G_{x}$ so that we also have
$G\cdot x=(G/G_{x})\cdot x=\hat{\psi}_{x}(G/G_{x})\subset\hat{\psi}_{x}(G/G_{X})$.
So the image of $\hat{\psi}_{x}$ coincides with the orbit through
$x$, and we have concluded that the orbit is closed. Moreover, since
every proper map is a finite morphism onto its image, $\hat{\psi}_{x}:G/G_{X}\to G\cdot x$
is a finite morphism. Its fibers are the stabilizer of the reduced
action, that is, $(G/G_{X})_{x}$. This means that $(G/G_{X})_{x}=G_{x}/G_{X}$
is finite, and we have concluded that $x$ is stable. 

Conversely, if $G\cdot x$ is closed and $G_{x}/G_{X}$ is finite,
then the induced morphism $\hat{\psi}_{x}:G/G_{X}\to G\cdot x$ is
finite, because it has finite fibers (see \cite[Lemma 3.17]{N}),
and therefore it is proper.
\end{proof}

\subsection{The Numerical Criterion}

The following notions extend, to any $G$-variety $X$, the definitions
given before for the action of $G$ on itself under conjugation (Definition
\ref{def:rays}).
\begin{defn}
Given a 1PS $\lambda\in Y(G)$ and a point $x\in X$, consider the
morphism $\lambda_{x}:k^{\times}\to X$, defined by \[
\lambda_{x}(t):=\lambda(t)\cdot x,\quad t\in k^{\times}.\]
It will be called the \emph{$\lambda$-ray through $x$}. We say that
$\lim_{t\to0}\lambda_{x}(t)$ exists when $\lambda_{x}$ can be extended
to a morphism $\bar{\lambda}_{x}:k\to X$ , and write \[
\lambda^{+}x=\lim_{t\to0}\lambda_{x}(t)\]
for the unique element $\bar{\lambda}_{x}(0)\in X$. As before, whenever
a formula with limits is written, we are assuming that they exist.
\end{defn}
Consider the natural map $\pi_{*}:Y(G)\to Y(G/G_{X})$ defined by
composition with the canonical projection $\pi:G\to G/G_{X}$. The
following are easily deduced properties of $\lambda$-rays.
\begin{prop}
\label{pro:rays} Let $X,Y$ be affine $G$-varieties.

(i) Let $f:X\to Y$ be a $G$-morphism (ie, a morphism that is $G$-equivariant),
$x\in X$, and $\lambda\in Y(G)$. Then, the $\lambda$-rays in $X$
and $Y$ (through $x$ and $f(x)$, respectively) are related by $\lambda_{f(x)}=f\circ\lambda_{x}$.

(ii) If $f:X\to Y$ is an inclusion of affine $G$-varieties, $\lambda^{+}x$
exists in $X$ if and only if $\lambda^{+}(f(x))$ exists in $Y$.

(iii) We can define a $\lambda$-ray through $x$, for each $\lambda\in Y(G/G_{X})$,
since $G_{X}\subset G_{x}$ for all $x\in X$. Then, $\mu^{+}x$ exists,
$\mu\in Y(G)$, if and only if $(\pi_{*}\mu)^{+}x$ exists.
\end{prop}
%
{}

Using these notions, the numerical criterion of Mumford (\cite{MumFoKir})
can be rephrased in the following way. 
\begin{thm}
\label{thm:Mumford-Crit}Let $X$ be a $G$-variety and $x\in X$.
Then, $x$ is properly stable for the $G/G_{X}$ action if and only
if for every non trivial 1PS $\lambda\in Y(G/G_{X})$, the limit $\lambda^{+}x$
does not exist.\end{thm}
\begin{proof}
This is the direct application of Mumford's numerical criterion to
the $G/G_{X}$ action, the reduced action on $X$ (See \cite{MumFoKir}).
\end{proof}
Obviously, given a 1PS $\lambda\in Y(G)$ contained in $G_{X}$ (that
is, such that $\lambda(k^{\times})\subseteq G_{X}$), the limit $\lambda^{+}x$
exists for all $x\in X$. To avoid these redundant one-parameter subgroups,
we introduce the following notion.
\begin{defn}
Let $x\in X$. We say that $x$ is\emph{ 1PS stable }if it is polystable
and for every 1PS $\lambda\in Y(G)$ not contained in $G_{X}$, the
limit $\lambda^{+}x$ does not exist.
\end{defn}
From Mumford's criterion above, one easily sees that if $x\in X$
is stable, then it is also 1PS stable. However, it turns out that
the two concepts are indeed equivalent, as follows.
\begin{thm}
\label{thm:new-Mumford-crit}Let $X$ be a $G$-variety and $x\in X$.
Then, the following are equivalent:

(1) The point $x$ is stable;

(2) The point $x$ is properly stable for the $G/G_{X}$ action;

(3) The point $x$ is 1PS stable for the $G/G_{X}$ action;

(4) The point $x$ is 1PS stable.\end{thm}
\begin{proof}
Let $x$ be properly stable for the $G/G_{X}$ action. Then, by Theorem
\ref{thm:Mumford-Crit}, this is equivalent to the non-existence of
the limit $\lambda^{+}x$, for all $\lambda\in Y(G/G_{X})$ nontrivial.
Since \[
(G/G_{X})_{X}=\bigcap_{x\in X}(G/G_{X})_{x}=\left(\bigcap_{x\in X}G_{x}\right)/G_{X}=\{e\},\]
this means that $x$ is 1PS stable for the $G/G_{X}$ action. So (2)
and (3) are equivalent. Now, suppose $x\in X$ is properly stable
for the action of $G/G_{X}$. Then, $NG_{x}=(G/G_{X})_{x}=G_{x}/G_{X}$
is finite, and so, $x$ is stable for the $G$-action, so (2) implies
(1). Conversely if $x\in X$ is stable, then the reduced orbit morphism
$\hat{\psi}_{x}:G/G_{X}\to X$ is proper, by Proposition \ref{pro:stableproper}.
This means that $x$ is properly stable for the reduced $G/G_{X}$
action. The equivalence of (1) and (4) is the statement of Theorem
\ref{thm:stable-1ps-stable} below.
\end{proof}
To proceed, we need the following lemma.
\begin{lem}
\label{lem:Tori} Let $f:G\to G'$ be a surjective morphism of algebraic
groups, and let $f_{*}:Y(G)\to Y(G')$ be the canonical push-forward
map. For every $\mu\in Y(G')$, there exists $\lambda\in Y(G)$ and
$p\in\mathbb{N}$ such that $f_{\ast}(\lambda)=\mu^{p}$.\end{lem}
\begin{proof}
Assume, first, that $G$ and $G'$ are tori. So there are $n,\: m\in\mathbb{N}$
with $m\leq n$ such that $G\cong(k^{\times})^{n}$ and $G'\cong(k^{\times})^{m}$.
Let $\mu\in Y(G')=\mathbb{Z}^{m}$, $\mu(t)=(t^{l_{1}},\cdots,t^{l_{m}})$
for every $t\in k^{\times}$ and some $(l_{1},\cdots,l_{m})\in\mathbb{Z}^{m}$.
The map $f:G\cong(k^{\times})^{n}\to G'\cong(k^{\times})^{m}$ applied
to $(t_{1},\cdots,t_{n})\in(k^{\times})^{n}$ is $f(t_{1},\cdots,t_{n})=(t_{1}^{p_{1}},\cdots,t_{m}^{p_{m}})$
for $(p_{1},\cdots,p_{m})\in\mathbb{Z}^{m}$. Let $p$ be the least
common multiple of $|p_{1}|,\cdots,|p_{m}|$. In particular, there
exist $\alpha_{1},\cdots,\alpha_{m}\in\mathbb{Z}$, so that $p=\alpha_{1}p_{1}=\cdots=\alpha_{m}p_{m}$.
We define $\lambda\in Y(G)$ by $\lambda(t)=(t^{\alpha_{1}l_{1}},\cdots,t^{\alpha_{m}l_{m}},t^{x_{m+1}},\cdots,t^{x_{n}})$,
for each $t\in k^{\times}$ and some $x_{m+1},\cdots,x_{n}\in\mathbb{Z}$.
It is easy to check that $f_{\ast}(\lambda)=\mu^{p}$, where $\mu^{p}(t)=(t^{pl_{1}},\cdots,t^{pl_{m}})$.

In the general case, let $\mu\in Y(G')$. Then, there is a maximal
torus $T$ of $G'$ so that $\mu(k^{\times})\subset T$. By \cite[22.6]{Bo},
there exists a maximal torus $S$ of $G$ such that $f(S)=T$. Now,
we apply the previous case to prove the statement.
\end{proof}
This allows to complete the proof of Theorem \ref{thm:new-Mumford-crit}%
\footnote{We thank A. Schmitt for suggesting the use of tori in showing the
equivalence of stability and 1PS stability.%
}.
\begin{thm}
\label{thm:stable-1ps-stable}Let $x\in X$. Then, $x$ is stable
if and only if $x$ is 1PS stable. \end{thm}
\begin{proof}
If $x$ is not 1PS stable, then there exists a 1PS $\lambda\in Y(G)$
such that $\lambda(k^{\times})\nsubseteq G_{X}$ and $\lambda^{+}x$
exists. So, $\tilde{\lambda}^{+}x$ exists as well, where $\tilde{\lambda}$
is the (non-trivial) push forward 1PS of $G/G_{X}$. Thus, $x$ is
not stable by Theorem \ref{thm:new-Mumford-crit}. 

Now assume that $x$ is not stable. By Theorem \ref{thm:new-Mumford-crit},
there exists a non-trivial $\mu\in Y(G/G_{X})$ such that $\mu^{+}x$
exists. By Lemma \ref{lem:Tori}, applied to the canonical surjection
$\pi:G\to G/G_{X}$, we have $\mu^{p}=\pi_{*}\lambda$ for some $\lambda\in Y(G)$
and $p\in\mathbb{N}$. Moreover $\mu^{p}$ is non trivial and $(\mu^{p})^{+}x$
exists (because $p\in\mathbb{N}$). So $\lambda(k^{\times})\nsubseteq G_{X}$
since $\mu$ is non-trivial. If the limit $\lambda^{+}x$ did not
exist, then $(\mu^{p})^{+}x$ would not exist as well, by Proposition
\ref{pro:rays}(ii), so we conclude that $\lambda^{+}x$ exists, and
so, $x$ is not 1PS stable.\end{proof}
\begin{cor}
Let $X$ be a $G$-variety with $G_{X}$ finite. Then all three stability
notions are equivalent (stable, properly stable and 1PS stable).\end{cor}
\begin{proof}
By definition, $x$ is stable if $x$ is polystable and $G_{x}/G_{X}$
is finite, but since $G_{X}$ is finite, this means that $G_{x}$
is finite, so $x$ is stable if and only if $x$ is properly stable.
The equivalence of stability and 1PS stability was established in
the previous Theorem.
\end{proof}
To summarize, we showed the following relations between stability
notions of a point $x$ in a general affine $G$-variety $X$: $x$
properly stable $\implies$$x$ stable $\Longleftrightarrow$$x$
1PS stable $\implies$ $x$ 1PS polystable $\Longleftrightarrow$
$x$ polystable.

\section{Stability and Polystability Criteria}

In this section, we prove our main results (in particular Theorem
\ref{thm:Main}) on the relation between polystability of a point
$x$ in an affine $G$-variety and the complete reducibility of a
subgroup of $G$, naturally associated to $x$. Along the way, we
prove an analogue of a Theorem of Mundet-Schmitt in the context of
affine $G$-varieties, and provide some examples of the constructions.

\subsection{1PS polystability}

Let $G$ be an affine reductive group and $X$ be an affine $G$-variety.
Recall that, for any $x\in X$, the closure of its orbit $\overline{G\cdot x}$
contains a unique closed orbit. The following well known statement
characterizes non-closed orbits.
\begin{thm}
\label{thm:[Kempf]} If $x\in X$ is such that $\overline{G\cdot x}\setminus G\cdot x$
is non-empty then there is a $y\in X$ in the unique closed orbit
of $\overline{G\cdot x}$ and a one parameter subgroup $\lambda\in Y(G)$
such that $\lambda^{+}x=\lim_{t\to0}\lambda(t)\cdot x=y$.\end{thm}
\begin{proof}
See \cite{Birkes} or \cite{Kempf78}.
\end{proof}
This result motivates the following definition introduced by J. Levy
(see \cite{Lev} and also \cite[Definition 3.8]{BMRT})
\begin{defn}
Let $x\in X$. We say that $x$ is \emph{1PS polystable} if $\lambda^{+}x\in G\cdot x$
for all $\lambda\in Y(G)$ such that $\lambda^{+}x$ exists.
\end{defn}
With this terminology, Theorem \ref{thm:[Kempf]} can be restated
as an equivalence between the notions of polystability and 1PS polystability.
\begin{prop}
Let $x\in X$. Then, $x$ is polystable if and only if $x$ is 1PS
polystable.\end{prop}
\begin{proof}
If $x$ is polystable, then $G\cdot x$ is closed and the orbit $G\cdot x$
is a subvariety of $X$. Thus, if $\lambda^{+}x$ exists in $X$,
for some $\lambda\in Y(G)$, then $\lambda^{+}x$ is in fact in $G\cdot x$
(see Proposition \ref{pro:rays}). So $x$ is 1PS polystable. Conversely,
if $x$ is not polystable, so that $\overline{G\cdot x}\setminus G\cdot x$
is non-empty, then by Theorem \ref{thm:[Kempf]}, there exists $\lambda\in Y(G)$
and $y\notin G\cdot x$ such that $y=\lambda^{+}x$, which, by definition,
implies that $x$ is not 1PS polystable.\end{proof}
\begin{rem}
In \cite[Corollary 4.11]{BMRT} it is shown that this Proposition
also holds for affine $G$-varieties defined over non-algebraically
closed perfect fields.
\end{rem}

\subsection{Opposite 1PS and the Theorem of Mundet-Schmitt}

Let $\lambda$ be a 1PS of $G$. Recall the notions of R-parabolic
and R-Levi subgroups, given in Definition \ref{def:parab}. 

Recall that, given $\lambda\in Y(G)$, one says that another 1PS $\tilde{\lambda}$
is opposite to $\lambda$, if $P(\lambda)\cap P(\tilde{\lambda})$
is a R-Levi subgroup of both $P(\lambda)$ and $P(\tilde{\lambda})$.
For later convenience, we record the following fact.
\begin{lem}
\label{lem:opposite}Let $\lambda,\tilde{\lambda}\in Y(G)$ be opposite
1PS lying in the same maximal torus of $G$. Then $P(\lambda^{-1})=P(\tilde{\lambda})$.\end{lem}
\begin{proof}
We will use two uniqueness results, valid also for non-connected reductive
$G$: (i) Given a R-parabolic subgroup $P$ and a maximal torus $T$
of $G$ contained in $P$, there is a unique R-Levi subgroup of $P$
containing $T$; (ii) Given a R-parabolic subgroup $P$ and one of
its R-Levi subgroups $L$, there is a unique opposite R-parabolic
$P'$ such that $P\cap P'=L$ (see \cite[Cor. 6.5, Lem. 6.11, resp.]{BMR}).

To start, let $T$ be a maximal torus containing both $\lambda$ and
$\tilde{\lambda}$. Then, $T\subset L(\lambda)$ and $T\subset L(\tilde{\lambda})$
so both $P(\lambda)$ and $P(\tilde{\lambda})$ contain $T$. Since
$\lambda$ and $\tilde{\lambda}$ are opposite, $P(\lambda)\cap P(\tilde{\lambda})$
is a R-Levi subgroup $L$ of both, and clearly $L$ contains $T$.
By the uniqueness result (i) above, and since $L(\lambda)$ is a R-Levi
of $P(\lambda)$ containing $T$, we conclude that $L=L(\lambda)$.
Finally, since both $P(\lambda^{-1})$ and $P(\tilde{\lambda})$ are
opposite to $P(\lambda)$, and $L(\lambda)$ is a R-Levi of $P(\lambda)$,
the uniqueness result (ii) implies that $P(\lambda^{-1})=P(\tilde{\lambda})$,
as wanted. \end{proof}
\begin{defn}
A subset $\Lambda\subset Y(G)$ is called \emph{symmetric} if for
any 1PS $\lambda\in\Lambda$ there exists another 1PS $\tilde{\lambda}\in\Lambda$,
which is opposite to $\lambda$.
\end{defn}
Given $x\in X$, we use the following notation: \[
\Lambda_{x}:=\{\lambda\in Y(G):\lambda^{+}x\mbox{ exists}\}.\]
Note that if $\lambda^{+}x$ exists, and $g\in P(\lambda)$, then
a computation yields \[
\lambda^{+}(g\cdot x)=(\lambda^{+}g)\cdot(\lambda^{+}x).\]
In particular, if $u\in U(\lambda)$, then $\lambda^{+}(u\cdot x)=\lambda^{+}x$.
So, it is clear that $\Lambda_{x}=\Lambda_{u\cdot x}$ for any $u\in U(\lambda)$
such that $\lambda\in\Lambda_{x}$. The following properties of $\Lambda_{x}$
are straightforward.
\begin{prop}
\label{pro:lambda}Consider the natural inclusion $Y(G_{x})\subset Y(G)$.
For all $x\in X$, we have

(i) If $G\cdot x=x$ then $\Lambda_{x}=Y(G)$.

(ii) $Y(G_{X})\subset Y(G_{x})\subset\Lambda_{x}\subset Y(G)$.

(iii) If $x$ is stable, then $Y(G_{x})=\Lambda_{x}=Y(G_{X})$.
\end{prop}
Let $X$ and $Y$ be affine $G$-varieties, and $f:X\to Y$ be a $G$-equivariant
morphism. The following properties are also easily established.
\begin{prop}
\label{pro:inclusion}(i) For all $x\in X$, $\Lambda_{x}\subset\Lambda_{f(x)}$.

(ii) If $f:X\rightarrow Y$ is an inclusion, then $\Lambda_{x}=\Lambda_{f(x)}$
for all $x\in X$.

(iii) If $X\times Y$ is the product variety endowed with the natural
product $G$-action, then $\Lambda_{(x,y)}=\Lambda_{x}\cap\Lambda_{y}$
for every $(x,y)\in X\times Y$.
\end{prop}
Given a 1PS $\lambda\in Y(G)$ and $x\in X$, we have defined the
limit $\lambda^{+}x$. Similarly, we define $\lambda^{-}x$ as the
limit (when it exists) \[
\lambda^{-}x:=\lim_{t\to\infty}\lambda(t)\cdot x=\lim_{t\to0}\lambda^{-1}(t)\cdot x.\]
We now prove Theorem \ref{thm:M-S}.
\begin{thm}
\label{thm:poly-symm} $x\in X$ is polystable if and only if $\Lambda_{x}$
is symmetric.\end{thm}
\begin{proof}
Let $x\in X$ be polystable, and suppose that $\lambda\in\Lambda_{x}$.
Then $\lambda^{+}x$ exists, and because $G\cdot x$ is closed, $\lambda^{+}x\in G\cdot x$.
By the general theorems, \cite{BMRT} and references in \cite{MR},
$\lambda^{+}x=u\cdot x$ for some $u\in U(\lambda)$. Now, it is clear
that the limit $u\cdot x$ centralizes $\lambda$, so that $\lambda^{-1}\cdot(u\cdot x)=u\cdot x$.
This means that $\lambda^{-1}\in\Lambda_{u\cdot x}=\Lambda_{x}$.
Since $\lambda^{-1}$ and $\lambda$ are obviously opposite, we conclude
that $\Lambda_{x}$ is symmetric.

Conversely, if $G\cdot x$ is not closed, then by Theorem \ref{thm:[Kempf]},
there exists $\lambda\in Y(G)$ such that $\lambda^{+}x\in\overline{G\cdot x}\setminus G\cdot x$.
This implies that $\{\lambda(t)\cdot x:t\in k^{\times}\}\subset X$
is not a single point. Suppose to get a contradiction, that $\Lambda_{x}$
was symmetric: there is an opposite $\tilde{\lambda}\in Y(G)$ such
that $\tilde{\lambda}^{+}x$ exists. By the same argument in \cite[Appendix]{MR}
we can find a $u\in U(\lambda)$ such that $u\lambda(t)u^{-1}$ and
$\tilde{\lambda}(t)$ lie on the same maximal torus (and are still
opposite). So, by Lemma \ref{lem:opposite}, and because $u\in U(\lambda)\subset P(\lambda)$\[
P(\tilde{\lambda}^{-1})=P(u\lambda u^{-1})=P(\lambda).\]
Now, this means that $\tilde{\lambda}^{-1}\in\Lambda_{x}$, so that
\[
\tilde{\lambda}^{+}x=(\tilde{\lambda}^{-1})^{-}x=\lambda^{-}x.\]
Then, $\lambda_{x}:k\to X$ extends to a morphism $\lambda:\mathbb{P}_{k}^{1}\to X$
which has to be constant, contradicting the previous claim; we conclude
that $\Lambda_{x}$ is not symmetric.
\end{proof}
We now summarize the setting in Schmitt's appendix to Mundet's article
\cite{MR}. Suppose our $G$-variety is a $k$-vector space $V$.
We say that $v$ is \emph{semistable} if $0\notin\overline{G\cdot v}$.
Given a 1PS $\lambda\in Y(G)$ and a weight $n\in\mathbb{Z}$, define\[
V(n):=\left\{ v\in V|\lambda(t)\cdot v=t^{n}v\quad\forall t\in k^{\times}\right\} .\]
Then, we have a finite vector space decomposition in non-empty subspaces
$V=V(n_{1})\oplus\cdots\oplus V(n_{m})$ where $n_{1}<\cdots<n_{m}$.
With respect to this decomposition, write $v=v_{1}+\cdots+v_{m}$
and define\[
\mu(v,\lambda):=\min\left\{ n_{j}|v_{j}\neq0\right\} \]
Then, we have
\begin{lem}
Let $v$ be semistable. Then $\mu(v,\lambda)=0$ if and only if $\lambda^{+}v$
exists.\end{lem}
\begin{proof}
By construction\[
\lambda(t)\cdot v=\lambda(t)\cdot(v_{1}+\cdots+v_{m})=t^{n_{1}}v_{1}+\cdots+t^{n_{m}}v_{m},\]
with $n_{1}<\cdots<n_{m}$, so the limit $\lambda^{+}v=\lim_{t\to0}\lambda(t)\cdot v$
exists if and only if $\mu(v,\lambda)\geq0$. On the other hand, if
$\mu(v,\lambda)>0$, we would have $\lambda^{+}v=0$, contradicting
the semistability of $v$. \end{proof}
\begin{cor}
{[}Mundet-Schmitt, \cite[Thm A.5]{MR}{]} A semistable $v\in V$ is
polystable if and only if for every $\lambda\in Y(G)$ with $\mu(v,\lambda)=0$
there exists an opposite $\lambda'$ such that $\mu(v,\lambda')=0$.\end{cor}
\begin{proof}
In view of the previous Lemma, this is a simple restatement of Theorem
\ref{thm:poly-symm}.
\end{proof}

\subsection{Polystability and completely reducible subgroups}

In this subsection we provide a correspondence between the notion
of polystability (resp. stability) for points of the $G$-variety
$X$ and the notion of complete reducibility (resp. irreducibility)
for subgroups of $G$. In order to obtain this, we introduce the following
construction:

Given $\Lambda\subset Y(G)$ we define $H_{\Lambda}:=\bigcap_{\lambda\in\Lambda}P(\lambda)$;
For any such $\Lambda$, $H_{\Lambda}$ is a closed subgroup of $G$.

Let $S(G)$ be the set of closed subgroups of $G$. We consider the
map $\eta:X\to S(G)$ given by \[
x\mapsto H_{x}:=H_{\Lambda_{x}}\subset G.\]
The following properties are clear, in view of Propositions \ref{pro:lambda}
and \ref{pro:inclusion}. Again, $f:X\rightarrow Y$ denotes a $G$-equivariant
morphism of affine $G$-varieties, and $X\times Y$ the product $G$-variety.
\begin{prop}
\label{pro:eta(x)}Let $Z=Z(G)$ and $x\in X$.

(i) If $\Lambda_{x}\subset Y(Z)$, then $H_{x}=G$.

(ii) If $G_{x}=G$ then $H_{x}=Z$.

(iii) $Z\subset H_{x}\subset\bigcap_{\lambda\in Y(G_{x})}P(\lambda)$.

(iv) $H_{x}$ is irreducible if and only if $H_{x}=G$.

(v) $H_{f(x)}\subset H_{x}$ for all $x\in X$.

(vi) If $f:X\rightarrow Y$ is an inclusion, then $H_{x}=H_{f(x)}$
for all $x\in X$.

(vii) For every $(x,y)\in X\times Y$, $H_{x}\cup H_{y}\subset H_{(x,y)}$.
\end{prop}
Now, we give a name to the condition (\ref{eq:proper}) alluded in
the Introduction.
\begin{defn}
We say that $\eta:X\to S(G)$ is \emph{proper} if for every $x\in X$
and $\lambda\in Y(G)$, $H_{x}\subset P(\lambda)$ implies that $\lambda\in\Lambda_{x}$.
\end{defn}
Note that, for any action, with $\eta$ proper or not, the converse
implication is valid ie, if $\lambda\in\Lambda_{x}$ then $H_{x}\subset P(\lambda)$.
Indeed, assume that $\lambda\in\Lambda_{x}$, and that $g\in H_{x}=\bigcap_{\mu\in\Lambda_{x}}P(\mu)$;
then, $g\in P(\mu)$, for every $\mu\in\Lambda_{x}$, in particular
$g\in P(\lambda)$.
\begin{lem}
Let $\eta:X\to S(G)$ be proper, $x\in X$ and $\lambda\in Y(G)$.
Then $H_{x}\subset P(\lambda)$ if and only if $\lambda\in\Lambda_{x}$.\end{lem}
\begin{prop}
\label{pro:symCR}Let $\eta$ be proper and $x\in X$. If $\Lambda_{x}$
is symmetric then $\eta(x)=H_{x}$ is completely reducible in $G$.\end{prop}
\begin{proof}
Suppose that $\eta(x)=H_{x}=\bigcap_{\lambda\in\Lambda_{x}}P(\lambda)\subset P(\mu)$
for some $\mu\in Y(G)$. So, for every $g\in H_{x}$, there exists
$\mu^{+}g$. As $\eta$ is proper, $\mu^{+}x$ exists, then $\mu\in\Lambda_{x}$.
We have that $\Lambda_{x}$ is symmetric, this implies that there
is $\tilde{\mu}\in\Lambda_{x}$ opposite to $\mu$. Thus $\eta(x)\subset P(\mu)\cap P(\tilde{\mu})=L$
with $L$ a R-Levi subgroup of $P(\mu)$. We can conclude that $\eta(x)=H_{x}$
is completely reducible in $G$.
\end{proof}
Now, we prove the first part of our main result.
\begin{thm}
\label{thm:polyCR}Let $\eta:X\to S(G)$ be proper. Then, $x\in X$
is polystable if and only if $H_{x}=\eta(x)$ is completely reducible
in $G$. \end{thm}
\begin{proof}
From Proposition \ref{pro:symCR} and Theorem \ref{thm:poly-symm},
$H_{x}$ is completely reducible in $G$, for any $x\in X$ which
is polystable. Conversely, suppose that $H_{x}\subset G$ is completely
reducible, and let $\lambda\in\Lambda_{x}$. Then $H_{x}=\bigcap_{\mu\in\Lambda_{x}}P(\mu)\subset P(\lambda)$,
by definition. Since $H_{x}$ is completely reducible, there is $\tilde{\lambda}$,
opposite to $\lambda$ such that $H_{x}\subset P(\tilde{\lambda})$.
Since $\eta$ is proper, this implies that $\tilde{\lambda}\in\Lambda_{x}$.
We have concluded that $\Lambda_{x}$ is symmetric, so that by Theorem
\ref{thm:poly-symm}, $x$ is polystable.\end{proof}
\begin{cor}
Let $\eta:X\to S(G)$ be proper and $x\in X$. Then $\Lambda_{x}$
is symmetric if and only if $H_{x}=\eta(x)$ is completely reducible
in $G$. \end{cor}
\begin{proof}
This is an immediate consequence of Theorems \ref{thm:poly-symm}
and \ref{thm:polyCR}.\end{proof}
\begin{defn}
We say that the action of $G$ on $X$ is \emph{central} (or $X$
is a \emph{central $G$-variety}) if $G_{X}=Z(G)$.\end{defn}
\begin{thm}
\label{thm:1PSstableIrred}Let $X$ be a $G$-variety and assume that
$Z(G)\subset G_{X}$. Let $\eta$ be proper and $x\in X$. Then

(1) If $x$ is stable, then $H_{x}=G$.

(2) Assume that $X$ is central. Then $x$ is stable if and only if
$H_{x}=G$.\end{thm}
\begin{proof}
By Theorem \ref{thm:stable-1ps-stable}, $x$ is stable in $X$ if
and only if it is 1PS stable, so that for every 1PS, $\lambda$ such
that $\lambda(k^{*})\nsubseteq G_{X}$ it does not exist $\lambda^{+}x$.
As $\eta$ is proper, for every such 1PS, $\lambda$, of $G$, $\lambda^{+}g$
does not exist for some $g\in\eta(x)=H_{x}$. But this is equivalent
to saying that $H_{x}\nsubseteq P(\lambda)$ for every 1PS, $\lambda$,
of $G$ which is not in $G_{X}$. Since $Z(G)\subset G_{X}$ this
means that $H_{x}$ is not in $P(\lambda)$ for every non-central
$\lambda$. Thus $H_{x}$, is not contained in any proper R-parabolic,
so that $H_{x}$ is irreducible in $G$. Then, by Proposition \ref{pro:eta(x)}(iv)
$H_{x}=G$, and we have shown (1). To prove (2) suppose $G_{X}=Z(G)$.
If $x$ is stable, by (1) $H_{x}=G$. Conversely, if $x$ is not stable,
then it is not 1PS stable which means that $\lambda^{+}x$ exists
for some $\lambda\in Y(G)$, with $\lambda(k^{\times})\nsubseteq G_{X}=Z(G)$.
So $\lambda\in\Lambda_{x}$, which by properness of $\eta$ means
$H_{x}\subset P(\lambda)$, a proper R-parabolic. So $H_{x}\neq G$.
\end{proof}
Theorems \ref{thm:polyCR} and \ref{thm:1PSstableIrred} finish the
proof of Theorem \ref{thm:Main}.
\begin{example}
\label{exa:SL2}Let $G=SL_{2}(k)$ and $X=M_{2}(k)$, the vector space
of $2\times2$ matrices, where $G$ acts by conjugation on $X$. By
general facts, any non-trivial 1PS in $G$ is conjugated to\[
\lambda_{n}(t):=\left(\begin{array}{cc}
t^{n} & 0\\
0 & t^{-n}\end{array}\right),\]
for some $n=1,2,...$. Let $\lambda=h\lambda_{n}h^{-1}$, for $h\in G$
and $x\in X$. Then, by definition, $\lambda^{+}x$ exists, if and
only if\begin{equation}
\lim_{t\to0}\lambda_{n}(t)h^{-1}xh\lambda_{n}(t)^{-1}\label{eq:limit-x}\end{equation}
exists, which implies, by a simple computation, that $h^{-1}xh\in U,$
where $U$ is the set of \emph{upper} triangular matrices in $M_{2}(k)$.
First, let us compute $\Lambda_{x}$ and $H_{x}=\eta(x)$ for $x$
in the form: \[
x=\left(\begin{array}{cc}
\alpha & 0\\
0 & \beta\end{array}\right)\in X.\]
Writing \begin{equation}
h=\left(\begin{array}{cc}
a & b\\
c & d\end{array}\right)\in SL_{2}(k),\label{eq:ac-zero}\end{equation}
another computation shows that, in the case that $x$ is generic,
that is $\alpha\neq\beta$, $h^{-1}xh\in U$ is equivalent to $ac=0$.
So, we conclude that \[
\Lambda_{x}=\{h\lambda_{n}h^{-1}:h\in H_{1}\}\cup\{h\lambda_{n}h^{-1}:h\in H_{2}\},\]
where $H_{1}$ is the set of elements $h\in SL_{2}(k)$ in the form
(\ref{eq:ac-zero}) such that $a=0$, and $H_{2}=U$ the set of those
$h$ with $c=0$. We note that the decomposition of $\Lambda_{x}$
above, corresponds to opposite 1PS, in the sense that, for $h\in H_{1}$
and $\tilde{h}\in H_{2}$ the 1PS $h\lambda_{n}h^{-1}$ and $\tilde{h}\lambda_{n}\tilde{h}^{-1}$
are opposite.

Now, let $h\in H_{1}$ be fixed. Computing as in (\ref{eq:limit-x}),
with $g\in G$ instead of $x\in X$, \[
P(h\lambda_{n}h^{-1})=\left\{ g\in G:h^{-1}gh\in U\cap SL_{2}(k)\right\} =U\cap SL_{2}(k)\]
and, in the case $h\in H_{2}$, we have\[
P(h\lambda_{n}h^{-1})=\left\{ g\in G:h^{-1}gh\in U\cap SL_{2}(k)\right\} =L\cap SL_{2}(k)\]
where $L$ is the set of \emph{lower} triangular matrices in $M_{2}(k)$.
Finally, then\[
H_{x}=\bigcap_{\lambda\in\Lambda_{x}}P(\lambda)=\bigcap_{h\in H_{1}}P(h\lambda_{n}h^{-1})\bigcap_{h\in H_{2}}P(h\lambda_{n}h^{-1})=U\cap L\cap SL_{2}(k)\]
is the subgroup of diagonal matrices in $SL_{2}(k)$, a completely
reducible subgroup in $SL_{2}(k)$.

When $x$ is scalar, it is easy to see that $\Lambda_{x}=Y(G)$ and
$H_{x}=\{e\}$. Finally, when $x$ is of the form \[
x=\left(\begin{array}{cc}
\alpha & \beta\\
0 & \alpha\end{array}\right)\in X,\]
for $\beta\neq0$, by similar computations we obtain $H_{x}=U\cap SL_{2}(k)$,
which is not completely reducible in $G$.

To conclude, because in this case the map $\eta$ is clearly proper
(as the action of $G$ on itself is the restriction of the action
of $G$ on $X$), this example serves as a simple illustration of
our main results, namely that the orbit $G\cdot x$ is closed if and
only if $H_{x}$ is completely reducible.
\end{example}
The above constructions can be interpreted in the language of buildings.
In fact, there is a natural way to pass from $Y(G)$ to the building
of $G$, under which $\Lambda_{x}$ corresponds to a convex subset
in this building (See \cite[§2]{MumFoKir}). This observation, for
which we thank G. Röhrle, will be explored in a later work.

\subsection{A sufficient criterion for properness of $\eta$}

Above, we have discussed actions with $\eta$ proper. However, it
is easy to find examples where $\eta$ is not proper. For instance,
consider $G=\mathbb{C}^{*}$ and $X=\mathbb{C}^{n}$, with the action
of $G$ on $X$ given by $t\cdot(x_{1},\cdots,x_{n})=(t^{-1}x_{1},\cdots,t^{-1}x_{n})$
(where $t\in G$ and $(x_{1},\cdots,x_{n})\in X$). Trivially, $H_{x}=G$
for all $x\in X$ because for every $\lambda\in Y(G)$, $P(\lambda)=G$.
But $\lambda\notin\Lambda_{x}$ for the identity character ($\lambda(t):=t$
for $t\in\mathbb{C}^{*}$) and every $x\in X\setminus\{0\}$.

Some easy properties of the properness notion are the following. Let
again $X$, $Y$ be affine $G$-varieties, and $X\times Y$ the product
$G$-variety, with respective maps $\eta_{X}$, $\eta_{Y}$ and $\eta_{X\times Y}$.
Let $f:X\to Y$ be a $G$-morphism.
\begin{prop}
\label{pro:functorial}(i) If $f:X\to Y$ is an inclusion and $\eta_{Y}$
is proper, then $\eta_{X}$ is also proper.

(ii) If $\eta_{X}$ and $\eta_{Y}$ are proper then $\eta_{X\times Y}$
is also proper.\end{prop}
\begin{proof}
(i) Let $x\in X$, $\lambda\in Y(G)$, and suppose that $H_{x}\subset P(\lambda)$.
By Propositions \ref{pro:inclusion}(i) and \ref{pro:eta(x)}(vi),
we have $\Lambda_{x}=\Lambda_{f(x)}$ and $H_{f(x)}=H_{x}$, so $H_{f(x)}\subset P(\lambda)$.
Since $\eta_{Y}$ is proper, by hypothesis, $\lambda^{+}f(x)$ exists
in $Y$ and so $\lambda\in\Lambda_{f(x)}=\Lambda_{x}$. Thus $\eta_{X}$
is proper.

(ii) Let $(x,y)\in X\times Y$ and $\lambda\in Y(G)$. Suppose that
$H_{(x,y)}\subset P(\lambda),$ by \ref{pro:eta(x)}(vii) $H_{x}\cup H_{y}\subset H_{(x,y)}$
so $H_{x}$ and $H_{y}$ are contained on $P(\lambda)$. By hypothesis
$\eta_{X}$ and $\eta_{Y}$ are proper then $\lambda\in\Lambda_{x}\cap\Lambda_{y}$.
Thus $\lambda\in\Lambda_{(x,y)}$ by Proposition \ref{pro:inclusion}.
Therefore $\eta_{X\times Y}:X\times Y\rightarrow S(G)$ is proper. 
\end{proof}
The following is a simple sufficient condition to obtain properness. 
\begin{lem}
\label{lem:psi}Let $\psi:X\to G$ be a map that verifies the following
condition: For every pair \textup{$x\in X$}, $\lambda\in Y(G)$,
$\lambda^{+}(\psi(x))$ exists if and only $\lambda^{+}x$ exists.
Then, $\eta$ is proper.\end{lem}
\begin{proof}
The condition implies that if $\lambda\in\Lambda_{x}$, then $\psi(x)\in P(\lambda)$,
so $\psi(x)\in H_{x}=\bigcap_{\lambda\in\Lambda_{x}}P(\lambda)$.
Now let $H_{x}\subset P(\mu)$, for some $\mu\in Y(G)$. Then, $\psi(x)\in H_{x}\subset P(\mu)$,
so that $\mu^{+}(\psi(x))$ exists. Then, the condition implies that
$\mu\in\Lambda_{x}$, so $\eta$ is proper.
\end{proof}
Such a $\psi$ does not exist in general (see the example in the beginning
of this subsection), but when it is a $G$-equivariant morphism, the
existence of $\lambda^{+}x$ implies that of $\lambda^{+}(\psi(x))$,
by a simple calculation.

We now apply this Lemma to the adjoint representation, restricting
to the case of complex affine reductive groups (i.e, to the field
$k=\mathbb{C}$). Let $\lie(G)$ denote the Lie algebra of $G$, which
is naturally a $G$-variety with the conjugation action of $G$, and
let $\mbox{exp}:\lie(G)\to G$ be the exponential map. Note however,
that $\exp$ is not an algebraic morphism.
\begin{lem}
Let $V$ be a finite dimensional complex vector space, and $G=GL(V)$
be the complex general linear group. Then, for all $v\in\lie(G)\cong\mbox{End}(V)$
and $\lambda\in Y(G)$, $\lambda^{+}\mbox{exp}(v)$ exists iff $\lambda^{+}v$
exists.\end{lem}
\begin{proof}
By choosing basis for $V$, we can assume that $\lambda(t)$ is represented
by a diagonal matrix $\lambda(t)=\mbox{diag}(t^{k_{1}},...,t^{k_{n}})$,
for some weights $k_{1},...,k_{n}\in\mathbb{Z}$ with $k_{1}\geq k_{2}\geq\cdots\geq k_{n}$.
Since $\lambda^{+}\mbox{exp}(v)$ equals $\lim_{t\to0}\lambda(t)\mbox{exp}(v)\lambda(t^{-1})$,
it exists if and only if the matrix representing $\mbox{exp}(v)$
has zero entries below the main diagonal. By an elementary calculation,
this is equivalent to the matrix $v\in\lie(G)$ having zero entries
below the main diagonal, which in turn means that $\lambda^{+}v$
exists.\end{proof}
\begin{cor}
\label{cor:adjoint}Let $G$ be a complex affine reductive group,
and let $\lie(G)$ be its Lie algebra. Then, $\exp:\lie(G)\to G$
verifies the condition in Lemma \ref{lem:psi}. Thus, $\eta$ is proper
for the adjoint representation.\end{cor}
\begin{proof}
Assume that $G$ is equivariantly embedded in $GL(V)$ for some vector
space $V$. Then, we have the commutative diagram, \[
\begin{array}{ccc}
\mbox{End}V & \stackrel{\exp}{\to} & GL(V)\\
\varphi\uparrow &  & \uparrow\varphi\\
\lie(G) & \stackrel{\exp}{\to} & G.\end{array}\]
Suppose that $\lambda^{+}(\exp(v))$ exists in $G$, for $v\in\lie(G)$
so that $\lambda^{+}(\varphi(\exp(v)))=\lambda^{+}(\mbox{exp}(\varphi(v)))$
exists in $GL(V)$. By the previous Lemma, this means that $\lambda^{+}\varphi(v)$
exists in $\mbox{End}V$, which implies that $\lambda^{+}v$ exists
in $\lie(G)$ (note that $\varphi$ are inclusions of closed varieties).
Similarly, one shows that the existence of $\lambda^{+}v$ implies
that of $\lambda^{+}(\exp(v))$.\end{proof}
\begin{rem}
\label{rem:thm1.4}There are many other examples of affine $G$-varieties
with $\eta$ proper. For instance, one obtains a new example by removing
some closed $G$-invariant algebraic subset from a $G$-variety with
$\eta$ proper. Naturally, it is more interesting to work in the opposite
direction: start from the proper case and extend the action of $G$
to a bigger affine variety, preserving properness. We plan to exploit
these natural questions in a future work.
\end{rem}

\section{Application to $G$-invariant subvarieties of $G^{N}$}

Let $G$ be an affine reductive group defined over $k$. In this section,
we apply the techniques of the last sections to a particularly interesting
class of $G$-varieties, namely to closed $G$-invariant subvarieties
of the variety of $N$-tuples in $G$. This class includes the representation
varieties of finitely generated groups, which are varieties of the
form $\hom(\Gamma,G)$, consisting of homomorphisms from a fixed finitely
generated group $\Gamma$ into $G$. Because of the natural conjugation
action, the variety $X=\hom(\Gamma,G)$ becomes a $G$-variety. An
application of these results is a restatement of the well known correspondence
between $G$-Higgs bundles over an algebraic curve and character varieties
of surface groups.

\subsection{Polystability for subvarieties of $G^{N}$}

Let $X$ be an affine closed $G$-subvariety of $G^{N}$, where $G^{N}$
is the $G$-variety on which $G$ acts component-wise by simultaneous
conjugation. Let $f:X\hookrightarrow G^{N}$ be the $G$-equivariant
inclusion. Fix $x\in X$ and let $f(x)=(f_{1}(x),\cdots,f_{N}(x))$.
Let $\phi_{x}$ be the subgroup of $G$ generated by the elements
$f_{1}(x),...,f_{N}(x)\in G$. 

Recall the construction of the map $\eta:X\rightarrow S(G),$ with
$\eta(x)=H_{x}=\bigcap_{\lambda\in\Lambda_{x}}P(\lambda)$ . We have:
\begin{prop}
\label{etaisproper-1} Let $X$ be an affine closed $G$-subvariety
of $G^{N}$.

(i) Let $x\in X$ and $\mu\in Y(G)$, then $\mu\in\Lambda_{x}$ if
and only if $\phi_{x}\subseteq P(\mu).$

(ii) Let $x\in X$ then $\phi_{x}\subset H_{x}.$

(iii) $\eta$ is proper.\end{prop}
\begin{proof}
To show (i), let $x\in X$ and $\mu\in Y(G)$. Consider $\mu_{x}:k^{\times}\rightarrow X$
and $\mu_{f(x)}:k^{\times}\rightarrow G^{N}$ the $\mu$-rays through
$x$ and $f(x)$, respectively. We have that $\mu_{f(x)}=f\circ\mu_{x}$
by Proposition \ref{pro:rays} (i).

Suppose that $\mu\in\Lambda_{x}$, meaning that $\mu^{+}x$ exists.
So if $\mu^{+}x$ exists then $\mu^{+}f(x)$ exists. This is equivalent
to the existence of $\mu^{+}f_{i}(x)$, for every $i=1,\cdots,N$.
But this one implies the existence of $\mu^{+}g$, for every $g\in\phi_{x}$.
Thus $\phi_{x}\subset P(\mu)$. 

Conversely, suppose now that $\phi_{x}\subseteq P(\mu)$, this implies,
in particular, that $\mu^{+}f_{i}(x)$ exists for every $i=1,\cdots,N$.
Thus, $\mu^{+}f(x)$ exists. Since $X$ is a closed subset of $G^{N}$,
we can guarantee the existence of $\mu^{+}x$ (see Proposition \ref{pro:rays}
(ii)). Thus if $\phi_{x}\subseteq P(\mu)$ then $\mu\in\Lambda_{x}$. 

To show (ii), using (i) we have $\phi_{x}\subseteq P(\mu)$ for every
$\mu\in\Lambda_{x}$ so $\phi_{x}\subset\bigcap_{\lambda\in\Lambda_{x}}P(\lambda)=H_{x}$.

Suppose that $H_{x}\subset P(\mu)$. Then, by (ii) $\phi_{x}\subset H_{x}\subset P(\mu)$,
which means that $\mu\in\Lambda_{x}$, according to (i). So, $\eta$
is proper.
\end{proof}
For this class of varieties we have the next properties.
\begin{prop}
\label{pro:CRrep-1}Let $x\in X$. Then $H_{x}$ is completely reducible
in $G$ if and only if $\phi_{x}$ is completely reducible in $G$.\end{prop}
\begin{proof}
By (ii) of Proposition \ref{etaisproper-1}, we have that $\phi_{x}\subset H_{x}$. 

Suppose that $H_{x}=\bigcap_{\lambda\in\Lambda_{x}}P(\lambda)$ is
completely reducible. If $\phi_{x}\subset P(\mu)$ for some $\mu\in Y(G)$,
then by (i) of Proposition \ref{etaisproper-1} $\mu\in\Lambda_{x}$
and so $\phi_{x}\subset P(\mu).$ As $H_{x}$ is completely reducible,
$H_{x}\subset L$, with $L$ a R-Levi subgroup of $P(\mu)$. Thus
$\phi_{x}\subset L$ and then $\phi_{x}$ is completely reducible
in $G$.

Conversely, suppose that $\phi_{x}$ is completely reducible in $G$
and $H_{x}=\bigcap_{\lambda\in\Lambda_{x}}P(\lambda)\subset P(\mu)$
for some $\mu\in Y(G)$. By (ii) of Proposition \ref{etaisproper-1},
$\phi_{x}\subset H_{x}$ then $\rho(\Gamma)\subset P(\mu)$, so by
(i) of Proposition \ref{etaisproper-1}, $\mu\in\Lambda_{x}$. On
the other hand, as $\phi_{x}$ is completely reducible in $G$, there
exists a R-Levi subgroup $L=L(u\mu u^{-1})$ of $P(\mu)$, with $u\in U(\mu)$,
such that $\rho(\Gamma)\subset L$. Now, we can find a R-parabolic
subgroup of $G$, $P$, opposite to $P(\mu)$ so that $P(\mu)\cap P=L$,
and $P=P(u\mu^{-1}u^{-1})$. Thus $\phi_{x}\subset P$, so again by
(i) of Proposition \ref{etaisproper-1} $u\mu^{-1}u^{-1}\in\Lambda_{x}$.
Then $H_{x}\subset P$ and as also $H_{x}\subset P(\mu)$ we get that
$H_{x}\subset L$. Thus $H_{x}$ is completely reducible in $G$.\end{proof}
\begin{prop}
\label{etairred=00003DGirred-1}Let $x\in X$. If $H_{x}=G$, then
$\phi_{x}$ is irreducible in $G$.\end{prop}
\begin{proof}
If $\phi_{x}$ is not irreducible in $G$, by definition, there exists
a non central $\mu\in Y(G)$ such that $\phi_{x}\subset P(\mu)$.
So by (i) of Proposition \ref{etaisproper-1}, $\mu\in\Lambda_{x}$,
thus $H_{x}\subset P(\mu)$, and we conclude that $H_{x}$ is not
irreducible.\end{proof}
\begin{prop}
\label{centralizer}For any $x\in X$, we have $G_{x}=Z_{G}(\phi_{x})$,
where $Z_{G}(\phi_{x})$ is the centralizer of $\phi_{x}$ and $Z(G)\subset G_{X}$.\end{prop}
\begin{proof}
The first statement is clear. To show the second one, first note that
the stabilizer of any $x\in X$ contains the center of the group.
Thus, $Z(G)\subset\bigcap_{x\in X}G_{x}=G_{X}$. 
\end{proof}
The proof of next result (Theorem \ref{thm:Main-Rich}) follows easily
from the work of several authors on $N$-tuples in algebraic groups
(see \cite[Theorems 3.6 and 4.1]{Rich}, \cite[Corollary 3.7]{BMR}
and \cite{Martin}). We include a self-contained proof for the convenience
of the reader. Recall that a $G$-variety is called central when $G_{X}=Z(G)$. 
\begin{thm}
\label{thm:polyCRrep-1}Let $x\in X\subset G^{N}$. Then 
\begin{enumerate}
\item $x$ is polystable if and only if $\phi_{x}$ is completely reducible.
\item $x$ is stable if and only if $\phi_{x}$ is irreducible.
\item If $X$ is central then $x$ is equicentral if and only if $\phi_{x}$
is isotropic.
\end{enumerate}
\end{thm}
\begin{proof}
(1) By Proposition \ref{pro:CRrep-1} and Theorem \ref{thm:polyCR}
we just need to see that $\eta$ is proper. But this is true by (iii)
of Proposition \ref{etaisproper-1}. 

(2) If $x$ is stable in $X$, as $Z(G)\subset G_{X}$ (Proposition
\ref{centralizer}) and $\eta$ is proper ((iii) of Proposition \ref{etaisproper-1}),
we can apply Theorem \ref{thm:1PSstableIrred} (2) and therefore $H_{x}$
is irreducible in $G$. Thus by Proposition \ref{etairred=00003DGirred-1},
$\phi_{x}$ is irreducible.

Conversely, if $\phi_{x}$ is irreducible in $G$, since $\phi_{x}$
is completely reducible in $G$, by (1) we get that $x$ is polystable,
that is, $G\cdot x$ is closed in $G$. Then by \cite[1.3.3]{Rich}
$G_{x}$ is reductive. If we suppose that $G_{x}/G_{X}$ is infinite,
this implies that $G_{x}/Z(G)$ is infinite, as $Z(G)\subset G_{X}$.
So there exists $\lambda\in Y(G_{x})$ such that $\lambda(k^{\times})\nsubseteq Z(G)$.
Then $P(\lambda)$ is a proper parabolic subgroup of $G$ containing
$\phi_{x}$. This contradicts the assumption, concluding that $G_{x}/G_{X}$
is finite and thus $x$ is stable.

(3) It is enough to show that $G_{x}=G_{X}$ if and only if $Z_{G}(\phi_{x})=Z(G)$.
This is immediate by the assumption $G_{X}=Z(G)$ and the fact proven
in Proposition \ref{centralizer} that $G_{x}=Z_{G}(\phi_{x})$.
\end{proof}
From Proposition \ref{pro:CRrep-1} and Theorem \ref{thm:polyCRrep-1}
(2), Theorem \ref{thm:Hx-phix} follows. Consider now, for $N\in\mathbb{N}$
and a complex reductive $G$, with Lie algebra $\lie(G)$, the vector
space $X=(\lie(G))^{N}$, with the diagonal adjoint action of $G$.
The next result follows easily from Richardson (\cite[Thm. 3.6 and 4.1]{Rich}),
being also a natural corollary of our approach.
\begin{thm}
Let $k=\mathbb{C}$ and $X$ be a closed $G$-invariant subvariety
of $(\lie(G)){}^{N}$. An element $x\in X$ is polystable (resp. stable)
if and only if $H_{x}$ is completely reducible in $G$ (resp. $H_{x}=G$).\end{thm}
\begin{proof}
Consider first the case $X=(\lie(G)){}^{N}$, in which $Z(G)=G_{X}$.
Then, the statement is an easy consequence of Corollary \ref{cor:adjoint}
and Proposition \ref{pro:functorial} (ii), in view of Theorem \ref{thm:Main}.
The case when $X\subset(\lie(G)){}^{N}$ then follows from Proposition
\ref{pro:functorial} (i).
\end{proof}

\subsection{Stability in representation varieties}

Let $\Gamma$ be a finitely generated group. In this subsection we
let $X$ be the $G$-representation variety of $\Gamma$, that is
$X=\hom(\Gamma,G)$, where the $G$-action is by conjugation. It is
an affine closed $G$-invariant subvariety of $G^{N}$, so the results
obtained in the last subsection can be immediately applied to this
case. We will phrase them using the following standard terminology.
\begin{defn}
Let $\rho\in\hom(\Gamma,G)$, and $\rho(\Gamma)$ be its image in
$G$. The representation $\rho$ is called:
\begin{enumerate}
\item \emph{irreducible}, if $\rho(\Gamma)$ is an irreducible subgroup
in $G$. 
\item \emph{reductive}, if $\rho(\Gamma)$ is completely reducible in $G$. 
\item \emph{good}, if it is reductive and $G_{\rho}=Z(G)$.
\end{enumerate}
\end{defn}
\begin{rem}
(1) Because of these definitions, all properties satisfied by irreducible
(resp. completely reducible) subgroups are satisfied by irreducible
(resp. reductive) representations. In particular, we have the analogues
of the Propositions in Subsection \ref{sub:Irreducible-and-completely}.(2)
Another definition of irreducible representation, in the case $k$
algebraically closed and $G$ connected, appears in the work of Ramanathan
(\cite{Ramanathan}) which is equivalent to our definition if we use
Proposition \ref{pro:irredRam} and \cite[Prop. 2.1.2]{Ramanathan}. 
\end{rem}
Suppose now $\Gamma=\left\langle \gamma_{1},\cdots,\gamma_{N}|r_{1},\cdots,r_{N}\right\rangle $
is a presentation of the group $\Gamma$, where $\gamma_{1},\cdots,\gamma_{N}$
are generators for some $N\in\mathbb{N}$. From this, we form the
canonical projection map\[
\pi:F_{N}\twoheadrightarrow\Gamma\]
where $F_{N}$ is a free group of rank $N$. This allows us to view
any representation variety $X=\hom(\Gamma,G)$, by composing any representation
with $\pi$, as a closed $G$-invariant affine subvariety of $G^{N}=\hom(F_{N},G)$.
This allows the direct application of the previous results to this
class of varieties.
\begin{rem}
Note that, however, not all $G$-invariant closed affine subvarieties
of $G^{N}$ are character varieties for some $\Gamma$, as simple
examples show.
\end{rem}
Recalling the notation of the previous subsection applied to this
particular case, we have $f:\hom(\Gamma,G)\hookrightarrow\hom(F_{N},G)$
the inclusion such that for $\rho\in\hom(\Gamma,G)$, $f(\rho)=(\rho(\gamma_{1}),\cdots,\rho(\gamma_{N}))=\phi\in\hom(F_{N},G)=G^{N}$
and $\phi_{\rho}=\rho(\Gamma)=\rho(\pi(F_{N}))=\phi(F_{N})$. By Proposition
\ref{centralizer}, $Z(G)\subseteq G_{X}$ and for each $\rho\in\hom(\Gamma,G)$,
$G_{\rho}=Z_{G}(\rho(\Gamma)$). So it is easy to notice that a representation
$\rho$ is good if and only if $\rho(\Gamma)$ is isotropic.

Because we are in the conditions of Proposition \ref{etaisproper-1}
with $X=\hom(\Gamma,G)$, we conclude the following.
\begin{prop}
\label{etaisproper} Let $\rho\in\hom(\Gamma,G)$. 

(i) Let $\mu\in Y(G)$; Then $\mu\in\Lambda_{\rho}$ if and only if
$\rho(\Gamma)\subseteq P(\mu).$ 

(ii) $\rho(\Gamma)\subset\eta(\rho)=H_{\rho}.$ 

(iii) The map $\eta$ is proper. 

(iv) $H_{\rho}$ is completely reducible in $G$ if and only if $\rho(\Gamma)$
is completely reducible in $G$.\end{prop}
\begin{proof}
As $\phi_{\rho}=\rho(\Gamma)$, all claims follow from Propositions
\ref{etaisproper-1} and \ref{pro:CRrep-1}. 
\end{proof}
As a consequence of Theorem \ref{thm:polyCRrep-1}, we obtain an equivalence
between stability and irreducibility notions.
\begin{prop}
\label{etairred=00003DGirred-2}\label{thm:polyCRrep} Let $X=\hom(\Gamma,G)$
and $\rho\in X$. We have:

(i) $\rho$ is polystable if and only if $\rho$ is reductive

(ii) If $\eta(\rho)$ is irreducible in $G$ then $\rho$ is irreducible.

(iii) $\rho$ is stable if and only if $\rho$ is irreducible.

(iv) If $X$ is central, then $\rho$ is equicentral if and only if
$\rho$ is good.\end{prop}
\begin{proof}
To show (i), (iii) and (iv), we use again the equality $\rho(\Gamma)=\phi_{\rho}$
and Theorem \ref{thm:polyCRrep-1}. To prove (ii) note that showing
irreducibility of $\rho$ is the same proving that $\rho(\Gamma)$
is irreducible in $G$. We know that $\rho(\Gamma)=\phi_{\rho}$,
so if $H_{\rho}$ is irreducible, by Proposition \ref{etairred=00003DGirred-1}
we get that $\rho(\Gamma)=\phi_{\rho}$ is irreducible. The last statement
follows from (i) and Theorem \ref{thm:polyCRrep-1}(3).\end{proof}
\begin{rem}
\label{rem:central-rep-var}It is very common for representation varieties
to be central. A sufficient condition is that for every $g\in G$,
there exists a $\rho\in\hom(\Gamma,G)$ and an element $\gamma$ of
$\Gamma$ such that $\rho(\gamma)=g$, as one easily checks.
\end{rem}
We therefore obtained the following table that summarizes the situation
for a character variety $X=\hom(\Gamma,G)$. In each line we have
equivalences, and at the bottom line, the {*} indicates that we require
that $X$ is central: $G_{X}=Z(G)$.

\medskip{}

\begin{center}
{\footnotesize }\begin{tabular}{|c|c|c|}
\hline 
{\footnotesize Representation $\rho$} & {\footnotesize Image Subgroup $H=\rho(\Gamma)$} & {\footnotesize Affine GIT $\rho\in X$}\tabularnewline
\hline
\hline 
{\footnotesize reductive} & {\footnotesize completely reducible (CR)} & {\footnotesize polystable}\tabularnewline
\hline 
{\footnotesize reductive, $\lie(G_{\rho})=\lie(Z)$} & {\footnotesize irreducible} & {\footnotesize stable}\tabularnewline
\hline 
{\footnotesize good (reductive, $G_{\rho}=Z$)$^{*}$} & {\footnotesize isotropic (CR, $Z_{G}(H)=Z$)$^{*}$} & {\footnotesize equicentral$^{*}$}\tabularnewline
\hline
\end{tabular}
\par\end{center}{\footnotesize \par}

\subsection{Stability in character varieties and Higgs bundles}

In this last subsection, we present a restatement of the well known
correspondence between $G$-Higgs bundles and representations in terms
of the stability notions considered above, restricting to the case
$k=\mathbb{C}$. We start with a brief review the theory of $G$-Higgs
bundles, following mainly the article \cite{GGM} and references therein.

As mentioned in Section \ref{sub:Stability,proper-stability}, the
affine GIT quotient $X_{\Gamma}(G):=\hom(\Gamma,G)\quot G$ is an
affine algebraic set defined over $\mathbb{C}$ called the $G$-character
variety of $\Gamma$. It is an elementary fact that the properties
of irreducibility and complete reducibility are invariant by conjugation,
so they can be defined unambiguously for points of the character variety.

Let $\Sigma$ be a compact Riemann surface of genus $g\geq2$ and
let $\pi_{1}(\Sigma)$ be the fundamental group of $\Sigma$. The
space $\hom(\pi_{1}(\Sigma),G)$ is naturally identified with an algebraic
subvariety of $G^{2g}$.

Let us consider the universal central extension $\Gamma$ of $\pi_{1}(\Sigma)$
given by the short exact sequence $0\rightarrow\mathbb{Z}\rightarrow\Gamma\rightarrow\pi_{1}(\Sigma)\rightarrow1$
and $\Omega_{\Sigma}^{1}$ denote the canonical line bundle of $\Sigma$.

Recall that a $G$-Higgs bundle over $\Sigma$ is a pair $(E,\varphi)$
where $E$ is a holomorphic $G$-bundle over $\Sigma$ and $\varphi$
is a holomorphic section of $E(\mathfrak{g})\otimes\Omega_{\Sigma}^{1}$
with $E(\mathfrak{g})$ being the adjoint bundle of $E$ . In particular,
a $GL_{n}(\mathbb{C})$-Higgs bundle is what is called a Higgs bundle,
that it is a pair $(V,\phi)$ with $V$ a holomorphic rank $n$ vector
bundle and $\phi\in H^{0}(X,End(V)\otimes\Omega_{\Sigma}^{1})$.

There is a notion of (poly)stability for $G$-Higgs bundles. When
$G=GL_{n}(\mathbb{C})$ stability can be defined using what is called
the slope of a Higgs bundle. The notion was later generalized to the
case of $G$-Higgs bundles for $G$ complex (for the details, see
\cite{Simpson} or \cite{GGM}).

Consider now the moduli space of $G$-Higgs bundles over $X$, $\mathcal{M}_{X}(G)$,
that is the set of isomorphism classes of polystable $G$-Higgs bundles.
This moduli space has the structure of a complex algebraic variety,
for the details of its construction see Kobayashi (\cite{Kob}), Ramanathan
(\cite{Ramanathan1}), Simpson (\cite{SimpsonI,SimpsonII}) and Schmitt
(\cite{SchmittI,SchmittII}).

Denote by $X_{\Gamma}(G)$ the character variety of representations
$\rho:\Gamma\rightarrow G$ such that $\rho(J)\in Z(G)^{0}$, where
$J$ is a central element of $\Gamma$ and by $X_{\Gamma}^{r}(G)$
the moduli space of reductive representations in $X_{\Gamma}(G)$.
In the Appendix (Proposition \ref{pro:Zariski-closure}), we recall
that a representation $\rho$ is reductive (that is, $\rho(\Gamma)$
is completely reducible in $G$), if and only if the Zariski closure
of $\rho(\Gamma)$ in $G$ is a reductive group. 

Now, we are ready to state the non-abelian Hodge theorem. For details
see \cite{Hitchin,Simpson,GGM} and \cite{Corlette,Don}.
\begin{thm}
\label{thm:Hodge}Let $G$ be a connected reductive complex Lie group.
There is a homeomorphism $X_{\Gamma}^{r}(G)\simeq\mathcal{M}_{X}(G)$
under which a $G$-Higgs bundle is stable if and only if it comes
from an irreducible representation. Moreover, the $G$-Higgs bundle
is stable and simple if and only if the representation is reductive
with $Z_{G}(\rho)=Z(G)$.
\end{thm}
Note that this result generalizes both the Theorem of Narasimhan-Seshadri
which states that a vector bundle on $X$ of rank $n$ is stable if
and only if it is associated to a irreducible representation of $\Gamma$
in $U(n)$ and the Theorem of Ramanathan which states that a holomorphic
principal $G$-bundle is stable if and only if it is associated to
an irreducible unitary representation of $\Gamma$ in $G$, satisfying
$\rho(\Gamma)\subset K$, where $K$ is the maximal compact subgroup
of $G$. We observe that in \cite{GGM} this result is extended also
to the case of real reductive groups, which we do not consider here.

Using Theorem \ref{thm:polyCRrep}, we can restate the non-abelian
Hodge theorem as follows. Note that, by Theorem \ref{thm:polyCRrep},
and Proposition \ref{pro:Zariski-closure} we have the identity $X_{\Gamma}^{r}(G)=X_{\Gamma}^{ps}(G)$,
where $X_{\Gamma}^{ps}(G)$ denotes the moduli space of polystable
representations inside $X_{\Gamma}(G)$. 
\begin{cor}
Let $G$ be a complex connected reductive Lie group. There is a homeomorphism
$X_{\Gamma}^{ps}(G)\simeq\mathcal{M}_{X}(G)$ under which a $G$-Higgs
bundle is stable if and only if it comes from a stable representation.
Moreover, the $G$-Higgs bundle is stable and simple if and only if
the representation is good.
\end{cor}
As a concluding remark, we note that the setup in this article allows
the consideration of the moduli space $X_{\Gamma}^{r}(G)$ of polystable
representations of $\Gamma$ in a complex reductive group $G$ which
is not necessarily connected. As far as we know, Theorem \ref{thm:Hodge}
has not been proved in general for the case of $G$-Higgs bundles
with $G$ not connected (see however \cite{GGM2,O} for definitions
and some non-connected real reductive groups). On the other hand,
our results seem to indicate that this form of the non-abelian Hodge
theorem should still be valid in this case, for an adequate notion
of stability of $G$-Higgs bundles, defined in terms of reduction
of structure group to R-parabolics.

\appendix

\section{Completely Reducible, Irreducible and Isotropic subgroups}

Let $G$ be a reductive affine algebraic group (not necessarily irreducible)
defined over an algebraically closed field $k$ of characteristic
$0$. For the convenience of the reader, in this appendix we state
some known properties of completely reducible, irreducible, isotropic
and $Ad$-irreducible subgroups of $G$, and compare to other relevant
results in the literature (See \cite{BMR,BMR1,Martin,Sikora}).

\subsection{Alternative characterizations of complete reducibility and irreducibility}

We begin by recalling the following characterization of the centralizer
of an irreducible subgroup,
\begin{prop}
\label{pro:The-centralizer-of} The centralizer of an irreducible
subgroup of $G$ is a finite extension of $Z(G)$.\end{prop}
\begin{proof}
See \cite[Lemma 6.2]{Martin}. We give an alternative proof analogous
to the proof of Theorem \ref{thm:polyCRrep-1} (2). Let $H$ be an
irreducible subgroup of $G$ then $Z_{G}(H)$ is reductive (see \cite[Section 6.3]{BMR}).
Suppose that $Z_{G}(H)$ is an infinite extension of $Z(G)$. As $Z_{G}(H)$
is reductive, this implies that there exists $\lambda\in Y(Z_{G}(H))$
such that $\lambda(k^{\times})\nsubseteq Z(G)$. So $P(\lambda)$
is a proper R-parabolic of $G$ and $H\subset P(\lambda)$. Then $H$
is not irreducible.
\end{proof}
To see if a completely reducible subgroup is irreducible we have the
next criterion. We denote by $\lie(H)$ the Lie algebra of an algebraic
group $H$.
\begin{prop}
\label{pro:irreducibility-criterion}For a completely reducible subgroup
$H$ of $G$, the following claims are equivalent. 
\begin{enumerate}
\item $H$ is irreducible in $G$; 
\item $\dim Z_{G}(H)=\dim Z(G)$;
\item $\lie(Z_{G}(H))=\lie Z(G)$.
\end{enumerate}
\end{prop}
\begin{proof}
(1)$\Rightarrow$(2) See A.1.

$(2)\Rightarrow(1)$\cite[Corollary 16]{Sikora}.

(2)$\Leftrightarrow$(3) By Corollary 3.6 of \cite{Bo} we have $\dim\lie(Z_{G}(H))=\dim Z_{G}(H)$
and $\dim\lie(Z(G))=\dim Z(G)$ . So it is easy to see that (3)$\implies$(2).
For the other implication, if $\dim Z_{G}(H)=\dim Z(G)$ then $\dim\lie(Z_{G}(H))=\dim\lie(Z(G))$,
as $\lie(Z(G))$ is a $k$-vector subspace of $\lie(Z_{H}(G))$, they
must be equal.
\end{proof}
The next results provide other ways to check complete reducibility
and irreducibility.
\begin{prop}
\label{pro:Zariski-closure} Consider a subgroup $H\subset G$. Then
$H$ is completely reducible in $G$ if and only if $\overline{H}$,
the Zariski closure of $H$ in $G$, is a reductive group. %
{}\end{prop}
\begin{proof}
See \cite[Section 2.2]{BMR1}\end{proof}
\begin{prop}
\label{pro:irredRam} Assume that $G$ is connected. A subgroup $H\subset G$
is irreducible if and only if $H$ does not leave any parabolic subalgebra
of $\lie(G)$ invariant.\end{prop}
\begin{proof}
If $H$ is irreducible, suppose that there is a parabolic subalgebra,
$\mathfrak{p}$, of $\lie(G)$, such that, $H$ leaves it invariant,
i.e., $\Ad(H)(\mathfrak{p})\subset\mathfrak{p}$($\Ad$ is the adjoint
representation of $G$ in $\lie(G)$). Let $P$ be the parabolic group
whose Lie algebra is $\mathfrak{p}$. So $H\subset N_{G}(P)=P$ (where
$N_{G}(P)$ is the normalizer of $P$ in $G$ and it is equal to $P$
because this one is parabolic). This contradicts the assumption of
$H$ being irreducible.

If $H$ is not irreducible then there is a parabolic subgroup $P$
such that $H\subset P$. Let $\mathfrak{p}$ be the Lie algebra of
$P$ which is a parabolic subalgebra of $\lie(G)$. Via the adjoint
representation, $P$ acts on $\mathfrak{p}$. Now we just have to
prove that $\Ad(H)(\mathfrak{p})\subset\mathfrak{p}$. We can suppose
that $G\subset GL(n)$ for some $n\in\mathbb{N}$, in this case $\Ad(g)(Y)=gYg^{-1}$
for every $g\in G$ and $Y\in\lie(G)$. So if we restrict to $g\in H$
and $Y\in\mathfrak{p}$, as $H\subset P$, $gYg^{-1}\in\mathfrak{p}.$
Thus $H$ leaves invariant a parabolic subalgebra of $\lie(G)$. 
\end{proof}

\subsection{\label{sec:appen}Irreducible, $Ad$-irreducible and isotropic subgroups.}
\begin{defn}
A completely reducible subgroup $H\subset G$ is called \emph{isotropic}
in $G$ if $Z_{G}(H)=Z(G)$.\end{defn}
\begin{rem}
For $\rho\in\hom(\Gamma,G)$ is trivial to see that $\rho$ is good
in $\hom(\Gamma,G)$ if and only if $\rho(\Gamma)$ is isotropic in
$G$. 
\end{rem}
By Proposition \ref{pro:irreducibility-criterion}, we can see that
if $H$ is isotropic, then it is irreducible. But the converse is
not true, as the following example shows.
\begin{example}
\label{exa:ad-irreducible} Let $G=PGL_{2}(\mathbb{C})$ and $H$
be the subgroup of $PGL_{2}(\mathbb{C})$ generated by $\gamma_{1}:=\left(\begin{array}{cc}
i & 0\\
0 & -i\end{array}\right)$ and $\gamma_{2}:=\left(\begin{array}{cc}
0 & 1\\
-1 & 0\end{array}\right)$. The subgroup $H$ is reductive because it is finite and by Proposition
\ref{pro:Zariski-closure} it is completely reducible. We can prove
that $H=Z_{PGL_{2}(\mathbb{C})}(H)\neq Z(PGL_{2}(\mathbb{C}))=1$,
to see this it is enough to check that all $h\in H$ centralizes $\gamma_{1}$
and $\gamma_{2}$. Thus, $H$ is not isotropic, but it is irreducible
because $\dim Z_{PGL_{2}(\mathbb{C})}(H)=\dim Z(PGL_{2}(\mathbb{C}))=0$
by Proposition \ref{pro:irreducibility-criterion}.\end{example}
\begin{rem}
In \cite{HP} for the case $G=PSL_{2}(\mathbb{C})$ and in \cite{Sikora}
for $G$ an algebraic group over $\mathbb{C}$, it is also introduced
the definition of $Ad$-irreducible subgroup of $G$ which is a subgroup
whose image by the adjoint representation is irreducible in $GL(\lie(G))$.
By \cite[Proposition 13]{Sikora} every $Ad$-irreducible subgroup
of $G$ is isotropic.
\end{rem}
In the remainder of the appendix, we will see that the notion of irreducible
subgroup is unaffected when we restrict the group to its semisimple
quotient. The same will hold for the notion of completely reducible
subgroup. On the contrary, this will not hold for the notion of a
isotropic subgroup.
\begin{prop}
\label{homored}Let $\phi:G_{1}\rightarrow G_{2}$ be a homomorphism
of reductive algebraic groups.

(i) If $H$ is a completely reducible subgroup of $G_{1}$ then $\phi(H)$
is a completely reducible subgroup of $G_{2}.$

(ii) Suppose that $\phi$ is an epimorphism, then if $H$ is an irreducible
subgroup of $G_{1}$ then $\phi(H)$ is an irreducible subgroup of
$G_{2}.$ \end{prop}
\begin{proof}
See \cite[6.2]{BMR}. 
\end{proof}
As $G$ is reductive, we know that $G/Z(G)$ is a semisimple group
. Let $\phi:G\rightarrow G/Z(G)$ be the canonical projection. We
have the following results.
\begin{prop}
\label{compredsemi-1}Let $H\subset G$ be a subgroup of $G$. Then

(i) $H$ is completely reducible if and only if $\phi(H)$ is completely
reducible.

(ii) $H$ is irreducible if and only if $\phi(H)$ is irreducible.\end{prop}
\begin{proof}
For the proof of (i) and the implication (ii) $\Rightarrow$, see
\cite[6.2]{BMR}.

To show (ii)$\Leftarrow$, suppose $\phi(H)$ is irreducible. If $H\subset P$
for $P$ a proper R-parabolic subgroup of $G$, by \cite[Lemma 6.15]{BMR}
$\phi(P)$ is a R-parabolic subgroup of $G/Z(G)$. 

If $\phi(P)=G/Z(G)=\phi(G)$, as $Z(G)\subset P$ we have that $P/Z(G)=G/Z(G)$
and then $P=G$. So $\phi(P)$ is proper. And $\phi(H)\subset\phi(P)$
which is impossible because $\phi(H)$ is irreducible.
\end{proof}
By Proposition \ref{pro:The-centralizer-of}, if a subgroup $H$ of
$G$ is irreducible then the centralizer of $H$ in $G$, $Z_{G}(H)$,
is a finite extension of the center of $G$, $Z(G)$. Using Proposition
\ref{compredsemi-1}(ii), we see that $Z_{G/Z(G)}(\phi(H))$ is also
a finite extension of the center of $G/Z(G)$. The same does not happen
for the property $Z_{G}(H)=Z(G)$, as we see next.
\begin{example}
\label{exa:semigood}Let $G=GL_{2}(\mathbb{C})$ and $H$ be the subgroup
of $GL_{2}(\mathbb{C})$ generated by the matrices $\gamma_{1}$ and
$\gamma_{2}$ of Example \ref{exa:ad-irreducible}. The subgroup $H$
is reductive because it is finite and by Proposition \ref{pro:Zariski-closure}
it is completely reducible. We can prove that $Z_{GL_{2}(\mathbb{C})}(H)=Z(GL_{2}(\mathbb{C}))=\mathbb{C}^{*}$,
so that $H$ is isotropic. However, consider now the image $\phi(H)$
of $H$ under the quotient $\phi:G\to PGL_{2}(\mathbb{C})=G/Z(G)$.
Using example \ref{exa:ad-irreducible} we know that $\phi(H)$ is
not isotropic because $Z_{PGL_{2}(\mathbb{C})}(\phi(H))\neq Z(PGL_{2}(\mathbb{C}))=1$.
\end{example}

\section*{Acknowledgments}

We thank G. Röhrle for some useful references to previous literature
and their relation to this work, and the anonymous referee for useful
comments leading to a few more statements and an improved presentation.
We also thank I. Biswas, S. Lawton and A. Schmitt for interesting
conversations on these topics, and the Isaac Newton Mathematics Institute,
where some parts of this article were produced during the {}``Moduli
Spaces'' Programme, for the hospitality. The authors were partially
supported by FCT (Portugal) through Project PTDC/MAT/099275/2008.
\bibliographystyle{amsalpha}
\addcontentsline{toc}{section}{\refname}\bibliography{StableRepsHiggs}

\end{document}